\documentclass[10pt,reqno]{amsart}
\usepackage{amscd,amsmath,amssymb,amsthm,}
\title{Non-abelian extensions of minimal rotations}
\author{Ulrich Hab\"ock}
\author{Vyacheslav Kulagin}
\email{\{ulrich.haboeck@univie.ac.at\}\{skulagin@rambler.ru\}}
\address{Ulrich Hab\"ock, Institute of discrete Mathematics and Geometry, TU Vienna, Wiedner Hauptstrasse 8, A1040-Wien, Austria.}
\address{Vyacheslav Kulagin, Institute for Low Temperature Physics and Engineering, 47 Lenin Ave., Kharkov 61103, Ukraine.}
\date{April 5, 2008}

\thanks{The first author was partially supported by the FWF research project S9612-N13.}
\thanks{The second author was supported by grant INTAS YSF-05-109-5200.}

\newtheorem{theo}{Theorem}[section]
\newtheorem{cor}[theo]{Corollary}
\newtheorem{lem}[theo]{Lemma}
\newtheorem{prop}[theo]{Proposition}

\theoremstyle{definition}
\newtheorem{defi}[theo]{Definition}
\newtheorem{open question}[theo]{Open question}
\theoremstyle{remark}
\newtheorem{rem}[theo]{Remark}

\newcommand{\Z}{\mathbb Z}
\newcommand{\R}{\mathbb R}
\newcommand{\mc}{\mathcal}
\newcommand{\mf}{\mathfrak}
\DeclareMathOperator{\Aut}{\it{Aut}}

\DeclareMathOperator{\Ad}{\it{Ad}}
\DeclareMathOperator{\ad}{\it{ad}}
\DeclareMathOperator{\GL}{\it{GL}}

\begin{document}\allowdisplaybreaks\frenchspacing
\begin{abstract}
We consider continuous extensions of minimal rotations on a locally connected compact group X by cocycles taking values in locally compact Lie groups and prove regularity (i.e. the existence of orbit closures which project onto the whole basis X) in certain special situations beyond the nilpotent case \cite{nilpotent}.
We further discuss an open question on cocycles acting on homogeneous spaces which seems to be the missing key for a general regularity theorem.
\end{abstract}

\maketitle

\section{Introduction}

Let $T$ be a minimal homeomorphism of a compact metric space $X$ and $G$ be a locally compact metrisable group.
Any continuous function $f:X\longrightarrow G$ defines an extension $T_f$ of $T$ via the equation
\begin{equation*}
T_f^n (x,g) = \big(Tx, f(n,x)\cdot g\big),
\end{equation*}
for every $x\in X$, $g\in G$ and $n\in\Z$, where $f(n,x)$ is the \emph{cocycle} generated by $f$, i.e.
\begin{equation*}
f(n,x)=
\begin{cases}
f(T^{n-1} x)\cdots f(T x)\cdot f(x) & \textup{if}\enspace n\geq 1, \\
e & \textup{if}\enspace n=0, \\
f(-n,T^n x)^{-1} & \textup{if}\enspace n < 0,
\end{cases}
\end{equation*}
with $e$ being the identity in $G$.
In this paper we investigate the problem of regularity of such an extension, i.e.
to ask whether there exist orbit closures which project onto the whole basis $X$ (such orbit closures are called \emph{surjective}). 
It is known that for arbitrary base transformations $T$ the existence of such orbit
closures might fail, see \cite{Lemanczyk} (this corresponds to the situation of type $III_0$ cocycles in the classical abelian case).
However if $T$ is a minimal rotation on a locally connected group $X$ then every topologically recurrent cocycle with values in a nilpotent locally compact group $G$ does admit surjective orbit closures and the entire product space $X\times G$ (or in geometric terminology the trivial
$G$-bundle) decomposes into such orbit closures (which are closed sub-bundles of $X\times G$), see \cite{nilpotent}. 
The essential idea involved goes back to G. Atkinson \cite{Atkinson} who proved regularity for
the case $G=\R^d$, and was generalised later by M. Lema\'nczyk and M. Mentzen \cite{Lemanczyk,Mentzen} to general locally compact abelian groups.
Before the present paper no regularity results beyond the nilpotent case where known, and our aim here is to develop methods which work in more general situations.


The difficulty in treating non-abelian (non-compact) extensions is that the (local) essential ranges
introduced in \cite{nilpotent},
\begin{equation}\label{essential range}
E_x(f)= \bigcap_{\mathcal U, V} \bigcup_{n\in\mathbb Z} \mathcal U \cap T^{-n}\mathcal U \cap \big\{y: f(n,y) \in V\cdot g \big\} ,
\end{equation}
where the intersection is taken over all open neighbourhoods $\mathcal U$ of $x$ and $V$ of the identity in $G$, alter along the orbits by conjugation:
\begin{equation*}
E_{T^n x}(f) = f(n,x) \cdot E_x(f) \cdot f(n,x)^{-1},
\end{equation*}
for all $x\in X$ and $n\in\Z$.
Furthermore, unlike in the abelian case, these essential ranges might not be subgroups of $G$ for points outside a dense $G_\delta$-set in $X$.
In what follows we show that understanding the behaviour of the identity component of $E_x$ under conjugation of the cocycle is crucial for regularity:
if $x$ is any point in $X$ and if the mapping
\[
H_{T^n x} = f(n,x) \cdot E_x^0(f) \cdot f(n,x)^{-1},
\]
which is only defined along the orbit of $x$ extends \emph{continuously} to the entire space $X$, then the transformation $T_f$ admits such a decomposition into surjective orbit closures. 
On the one hand this improves the key tool used in \cite{nilpotent}, and secondly it puts our attention more on the behaviour of these identity components under conjugation. 
This approach recalls the conjugacy problem of stabilizers for general Borel actions in S.G. Dani's paper \cite{Dani}, and in line with \cite{Dani} we show that the identity components of $E_x$ are conjugate on a dense $G_\delta$-set in $X$.
In some special situations we are able to prove that the identity components $E_x^0$
depend continuously on $x$ which implies regularity of the cocycle.
However, in general this issue is still open and is closely related with the following open question: 

\emph{Let $T_f$ be a continuous $G$-extension of a minimal group rotation $T$, and $H$ be any closed subgroup of $G$.
Suppose $C\subseteq X\times G/H$ is a compact $T_f$-orbit closure which projects injectively onto a dense $G_\delta$-subset of $X$ (which means, in particular, that $(C, T_f)$ is an almost one-to-one extension of the rotation\footnote{for the definition see \cite{Glasner}.}).
Is it true that then the projection $\pi:C\longrightarrow X$ is one-to-one on the whole set $C$?} 

This question was pointed out before in \cite{nilpotent}, but as its answer is positive for nilpotent groups $G$ we did not realise its importance at that time.

The paper is organised as follows:
first of all we review basic facts on cocycles taken from \cite{nilpotent}.
In Section \ref{s:Atkinson lemma} we prove the generalised Atkinson Lemma for general locally compact groups $G$ and draw some simple conclusions.
In Section \ref{s:Lie groups} we restrict our considerations to Lie groups, and adapt the results from \cite{Dani} to our setting in order to investigate the behaviour of the identity components $E_x^0$ under conjugation by the cocycle; we further discuss the importance of the above mentioned open question.
Finally, in the last section we show the existence of surjective orbit closures in the situation of semi-direct products $G=\R^d\rtimes \R$ where the action of $\R$ on $\R^d$ has no eigenvalue equal to one.
The proof presented there is alternative to the approach in Section \ref{s:Lie groups}.
However, it does not give a clearer picture of the general case; it is rather the simple group structure that allows us to reduce to situations that are easily understood.

It is worth to note that very likely all these results can be extended for a larger class of base transformations as is done in \cite{furstenberg} and \cite{furstenberg nilpotent}, but we will not focus on that issue in this paper.
The authors would like to thank Manfred Einsiedler and Klaus Schmidt, whose conversation has been very helpful.

\section{Basic facts and notions}\label{s:basic facts}

Let $T$ be a minimal homeomorphism of a compact metric space $X$ and $G$ a locally compact second countable (l.c.s.c.) group.
A cocycle $f(n,x)$ is said to be (topologically) \emph{recurrent} if for every open neighbourhood
$U$ of the identity in $G$ and every open set $\mc U\subseteq X$ there is an integer $n\neq 0$ so that
\begin{equation*}
T^{-n}\mc U \cap \mc U \cap \big\{x: f(n,x)\in U\big\}\neq\varnothing.
\end{equation*}
This property is equivalent to say $T_f$ being topologically conservative (or regionally recurrent in the terminology of \cite{GH}), i.e. for every open set $\mc O\subseteq X\times G$ there is an integer $n\neq 0$ so that $T_f^n(\mc O)\cap\mc O\neq\varnothing$.

The local essential range $E_x(f)$ defined by \eqref{essential range} is a closed subset of $G$ and it is symmetric, i.e. $E_x^{-1}(f) = E_x(f)$.
For every $x$ in $X$ the set
\begin{equation}\label{e:P_x}
P_x(f)=\Big\{ g\in G: (x,g) \in\overline{T_f^\Z (x,e)} \Big\}
\end{equation}
is a closed sub-semigroup of $G$.
We will simply write $E_x$ and $P_x$ whenever it is clear to which cocycle we refer.
It is shown in \cite[Proposition 1.7]{nilpotent} that the set
\begin{equation}\label{e:D(f)}
\mc D(f) =\big\{x\in X: E_x = P_x\big\}
\end{equation}
contains a dense $G_\delta$-set, thus it is non-meager in $X$.
Thus for every $x$ in $\mc D(f)$ the set $E_x$ is a closed symmetric sub-semigroup and hence a subgroup of $G$.

Recall that the essential ranges as well as the sub-semigroups $P_x$ satisfy the equation
\begin{equation}\label{e:essential ranges}
E_{T^n x} = f(n,x) \cdot E_x \cdot f(n,x)^{-1},
\end{equation}
for all $x\in X$ and $n\in\Z$, thus they are conjugate along orbits of $T$ \cite[Lemma 1.3]{nilpotent}.
The map $x\mapsto E_x$ is semi-continuous in the sense that if $x_n\rightarrow x$ and $g_n\in E_{x_n}$ converge to $g$, then $g\in E_x$.

If $H$ is a closed subgroup of $G$, then the action of $T_f$ (or its corresponding cocycle) on $X\times G/H$ is defined by setting
\[
T_f^n (x,g H) = \big(T^n x, f(n,x)\cdot g H\big).
\]
Any $T_f$-orbit closure in $X\times G/H$ is called \emph{surjective} if it projects onto $X$.
We shall make frequently use of the following lemma which is similar to \cite[Lemma 2.3]{nilpotent}.

\begin{lem}\label{l:surjective}
Let $C\subseteq X\times G/H$ be a $T_f$-invariant set which projects onto a non-meager set in $X$. 
Then there exists a compact set $K\subseteq G/H$ such that
$(X\times K)\cap C$ projects onto the whole set $X$.
\end{lem}
\begin{proof}
Choose a sequence $\{K_n\}_{n\geq 1}$ of compact subsets of $G/H$ such that $G/H = \bigcup_n K_n$.
Then the sets $K_n'=\pi_X\big((X\times K_n)\cap C\big)$, where $\pi_X$ is the projection onto $X$, are compact subsets of $X$ and their union
$\bigcup_{n\geq 1} K_n'$ is a non-meager set.
By Baire's category theorem there is an $m\geq 1$ such that $K_m'$ contains a non-empty open set $\mc U$ of $X$.
Since $T$ is minimal and $X$ is compact, $X=\bigcup_{n=1}^N T^{-n} (\mc U)$ for some $N\geq 1$ and
\[
\bigcup_{n=1}^N T_f^{-n} \big((X\times K_m) \cap C\big) = \bigcup_{n=1}^N T_f^{-n} (X\times K_m)  \cap C
\]
is a compact subset of $C$ that projects onto $X$.
\end{proof}

A cocycle $f$ is called \emph{regular} if its skew product transformation $T_f$  admits surjective orbit closures in $X\times G$.
By \cite[Theorem 2.1]{nilpotent} any surjective orbit closure $C$ is of the following form:
If we set
\[
H = \{ g\in G: C\cdot g^{-1} = C \},
\]
then $C/H$ is compact regarded as $T_f$-invariant subset of $X\times G/H$, and the restriction of $T_f$ to $C/H$ is minimal.
Moreover for every $x$ in $\mc D(f)$ the vertical section of $C$ consists of single coset of $H$ only:
there exists $g_x\in G$ such that
\[
C_x = \{g\in G: (x,g)\in C\} = g_x \cdot H.
\]
Thus the system $(C/H,T_f)$ is an almost one-to-one extension of $(X,T)$. It is
further shown that the map
\begin{equation*}
\gamma: \mc D(f)\longrightarrow G/H, \quad x\mapsto C_x = g_x \cdot H
\end{equation*}
is continuous and $E_x= g_x\cdot H \cdot g_x^{-1}$ for all $x\in\mc D(f)$ \cite[Theorem 2.2]{nilpotent}.

\begin{defi}
We call an orbit closure $C$ \emph{strongly regular} if it is surjective and \emph{every} vertical section $C_x$ as above consists of a single left coset of $H$.
A \emph{strongly regular cocycle} is a cocycle $f$  whose extension $T_f$ admits strongly regular orbit closures.
\end{defi}

\begin{rem}
It is shown in \cite[Theorem 3.1]{nilpotent} that every regular cocycle with values in a nilpotent group is strongly regular, but for general groups this issue is still open even for a rotation as a basis transformation $T$ (cf. the open question mentioned in the introduction).
\end{rem}

In other words a strongly regular orbit closure is a sub-bundle of $X\times G$. 
Note that then the entire product space (the trivial bundle) $X\times G$ decomposes into such $T_f$-invariant sub-bundles which are permuted via the right action of $G$ on $X\times G$ defined by $R_h (x,g)= (x,g\cdot h^{-1})$. 
For such orbit closures the above statements on $\gamma$ and $E_x$ remain to be true with $\mc D(f)$ replaced by $X$: for every $x$ in $X$ the vertical section $C_x=\big\{g\in G: (x,g)\in C\big\}$ consists of a single left coset of $H=\big\{g: R_g (C) = C\big\}$ and the mapping
\begin{equation*}
\gamma: X\longrightarrow G/H, \quad x\mapsto C_x = g_x \cdot H
\end{equation*}
is continuous on the whole set $X$.
It is easy to see that then $E_x= g_x \cdot H \cdot g_x^{-1}$ for every $x$ in $X$ (cf. the proof of \cite[Theorem 3.2]{nilpotent}).
Thus all essential ranges are subgroups conjugate to $H$, and if we identify $H^G$ the conjugacy class of $H$ with $G/N(H)$, where $N(H)$ is the normaliser of $H$, then
\[
\varphi: X\longrightarrow H^G, \quad x\mapsto E_x = g_x\cdot H \cdot g_x^{-1},
\]
is continuous.

Finally it should be noted that if $f$ is continuously cohomologous to a topological transitive cocycle taking values in a closed subgroup $H$ of $G$, then $f$ is strongly regular but not vice versa (if one does not allow discontinuities for the boundary function). 
More generally, if $b:X\longrightarrow G$ is continuous and the cocycle
\[
\tilde f(n, x) = b(T^n x) \cdot f(n, x) \cdot b(x)^{-1}
\]
is strongly regular, then $f$ is also strongly regular.


\section{The generalised Atkinson Lemma}\label{s:Atkinson lemma}

Let $\mc S(G)$ be the set of all closed subsets of $G$ equipped with the Fell topology (= projective limit of the Hausdorff topology on every compactum).
A basis for this topology is given by sets of the form
$\{S\in \mc S(G) : S\cap K\neq 0, S\cap O_i \neq 0 \text{ for } i=1,\ldots,k\}$,
where $K$ is any compact subset of $G$ and every $O_i$ is open.
It is well known that $\mc S(G)$ is compact and metrisable, and $\mc C(G)$ the space of all closed subgroups of $G$ is a closed subspace (see \cite{Fell}).
setting
A \emph{consistent selection} of subgroups $\{H_x\}_{x\in X}$ is a continuous mapping from $X$ into $\mc C(G)$ such that
\[
H_x\subseteq E_x
\]
for every $x$ in $X$, and which fulfills the consistency condition that
\begin{equation}\label{e:consistent selection}
H_{T^n x} = f(n,x)\cdot H_x \cdot f(n,x)^{-1},
\end{equation}
for every $x$ and $n\in\Z$.
\emph{In contrast to \cite{nilpotent} we do not assume that all $H_x$ belong to the same conjugacy class and assume continuity only with respect to the Fell topology.}

We will need the following auxiliary lemma on consistent selections:
\begin{lem}\label{l:cutting neighbourhood}
Let $\{H_x\}$ be a consistent selection for and let $U$ be any relatively compact open set in $G$.
Then the set $M_{U}= \{x\in X: E_x \cap \overline U H_x \setminus UH_x = \varnothing \}$
is open.
\end{lem}

\begin{proof}
We show that the complement of $\mc M_U$ is closed.
Indeed, suppose $x_k$ is a sequence of points converging to $x$ such that $E_{x_k}\cap \overline
U H_{x_k}\setminus U H_{x_k}\neq \varnothing$.
For any choice of relatively compact neighbourhoods $V$ and $W$ such that $\overline V\subseteq U$ and
$\overline U\subseteq W$ one can find points $z_k$ and $T^{n_k}z_k$ both converging to $x$ such that
\[
g_k = f(n_k,z_k) \in  \overline W H_{z_k}\setminus V H_{z_k}.
\]
Since $H_x$ depends continuously (with respect to the Fell topology) on $x$ we may assume without loss of generality that the points $z_k$ and $T^{n_k}z_k$ are from our dense set $\mc D(f)$, and therefore - after modifying the cocycle values along the essential ranges - the $g_k$ stay in
some fixed compactum. Thus the $g_k$ converge along some subsequence to some element $g$ which must be contained in the set $E_x\cap \overline W H_{x}\setminus V H_{x}$.
As $V$ and $W$ were arbitrary, this implies that $E_{x}\cap \overline U H_x\setminus U H_x\neq \varnothing$.
\end{proof}

We omit the proof of the following Lemma which is verbatim as the one for Lemma 4.3 in \cite{nilpotent}.
Their proof is in the same manner as the previous Lemma.

\begin{lem}\label{l:connectedness}
Let $U\subseteq G$ be an open subset and $C\subseteq G$ a compact subset.
Then for any fixed integer $n$ the sets $\{y\in X: f(n,y) \cdot H_y\cap U H_y\neq\varnothing\}$ and $\{y\in X: f(n,y) \cdot H_y\cap C H_y=\varnothing\}$ are both open.
\end{lem}

The following proposition which generalises a Lemma of G. Atkinson in \cite{Atkinson} will be the key for proving regularity of cocycles.
It is an improvement of the corresponding generalisation \cite[Proposition 4.4]{nilpotent} as we do only need `cutting neighbourhoods' at a single point in $X$; we moreover do not make any assumptions on the group $G$.

\begin{prop}\label{p:generalised Atkinson}
Suppose that $G$ is a l.c.s.c. group and $f:X\longrightarrow G$ is a recurrent cocycle over a minimal rotation $T$ on a locally connected compact group $X$, and let $\{H_x\}_{x\in X}$ be a consistent selection of subgroups. 
If there exists a point $x_0$ for which the group $H_{x_0}$ has a cutting neighbourhood in $E_{x_0}$, i.e. a relatively compact open neighbourhood $U$ of the identity such that
\[
E_{x_0}\cap \overline U H_{x_0}\setminus U H_{x_0} =\varnothing,
\]
Then the $T_f$-orbit closure of any point $(x,H_x)$ is a compact subset of $X\times G/H_x$.
\end{prop}

\begin{proof}
First of all note that it is sufficient to prove the existence of a single point $x$ such that the $T_f$-orbit closure of $(x,H_x)$ is compact, since this implies compactness of all other $T_f$-orbit closures.
Indeed, if $C$ is such a compact orbit closure then it projects onto the whole basis $X$.
Thus for every $y$ in $X$ there exists a $g\in G$ such that $(y,g H_x)\in C$ and therefore
we can find a sequence $\{n_k\}$, elements $h_k\in H_x$ and $g\in G$, such that $T^{n_k}x \rightarrow y$ and $f(n_k,x)\cdot h_k \rightarrow g$.
By continuity of the consistent selection we see that
\[
H_y= \lim_{k\rightarrow\infty} f(n_k,x) \cdot H_x\cdot f(n_k,x)^{-1} = g \cdot H_x \cdot g^{-1}.
\]
This shows that the compact $T_f$-orbit of $(y,g H_x)$ in $X\times G/H_x$ translates under the right translation by $g^{-1}$ to the $T_f$-orbit of $(y, H_y)$ in $X\times G/H_y$.
As a right translation is a homeomorphism the orbit closure of $(y,H_y)$ is also compact.

According to Lemma \ref{l:cutting neighbourhood} the set $\mc M_{U}= \big\{x\in X: E_x \cap \overline U H_x \setminus UH_x = \varnothing \big\}$ is open for every relatively compact open neighbourhood $U$, and therefore the $T$-invariant non-empty set $\mc M_{cut} = \bigcup_{U} \mc M_U$, where the union is taken over all relatively compact open neigbourhoods of the identity, is open too.
This yields $\mc M_{cut}=X$ and thus for every point $y$ in $X$ we can find a relatively compact cutting neighbourhood.

By recurrence both sets
\[
\mc R_\pm = \Big\{ x\in X: (x,e)\in \overline{T_f^{\Z_\pm} (x,e)} \Big\}\footnote{Here $\Z_+$ and $\Z_-$ denotes the set of all integers $>0$ and $<0$, respectively.}
\]
are comeager subsets of $X$, and so is the intersection $\mc R_+ \cap \mc R_- \cap \mc D(f)$.
Choose any point $x$ from this non-empty intersection and set
\[
C = \overline{T_f^{\Z} (x,H_x)}.
\]

Let $(y, g H_x)$ be any point belonging to the orbit closure $C$.
By our choice of $x$ there exists an increasing sequence of integers $n_k>0$ such that $(y,g)= \lim_{k\rightarrow\infty} T_f^{n_k} (x,e)$.
As above, we conclude that $H_y= g \cdot H_x \cdot g^{-1}$.
Let $U$ be a relatively compact cutting neighbourhood for $H_y$ in $E_y$.
Since $\mc M_U$ is open we can choose a connected open neighbourhood $\mc U$ of $y$ such that
\[
f(n,z) \in \begin{cases} U H_z,  \\ G\setminus\overline U H_z \end{cases}\text{ whenever } z \text{ and } T^n z \in \mc U.
\]
By convergence of $T_f^{n_k}(x,e)$ to $(y,g)$ we can find an integer $k_0$ such that
$z= T^{k_0}x \in\mc U$ and
\[
f(n_{k_0}- n_k, z)\in  U H_z
\]
for all $k\geq k_0$.
As the neighbourhood $\mc U$ is connected 
it follows from Lemma \ref{l:connectedness} that the same is true with $y$
replaced by $z$. Therefore all the cocycle values
\[
\begin{aligned}
f(-n_k, y)\cdot g & = f(-n_{k_0}, T^{-(n_k - n_{k_0})}y)\cdot f\left(n_{k_0} - n_k,y\right) \cdot g \\
&\in f(-n_{k_0}, T^{-(n_k - n_{k_0})}y)\cdot U \cdot \underbrace{H_y \cdot g}_{=g \cdot H_x },
\end{aligned}
\]
and therefore stay within some compact subset of $G/H_x$ as $k\rightarrow\infty$. Together with the fact that $T^{-n_k}y$ converges back to $x$ it implies that the $T_f^{-n_k} (y,gH_x)$ converge along some subsequence to $(x,g'H_x)$ with $g'$ in $P_x$.
In other words, $(x,g'H_x)$ is in the negative orbit closure of $(y,gH_x)$.
In the same manner one sees that also $(x,H_x)$ is in the negative orbit closure of $(x,g'H_x)$ and
therefore it is contained in the negative orbit closure of $(y,gH_x)$.

By the same argument one shows that $(x,H_x)$ is also contained in the positive orbit closure of $(y,gH_x)$. Together with recurrence, we conclude from \cite[Theorem 7.05]{GH} that $T_f$ restricted to $C$ is almost periodic and therefore $C$ is compact.
\end{proof}

\begin{prop}\label{p:regular orbit closure}
Under the same assumptions as for Proposition \ref{p:generalised Atkinson}, the $T_f$-orbit closure $C$ of any point $(x,e)$, with $x$ from the set $\mc D(f)$, is strongly regular.
Moreover $H_x$ is co-compact and normal in $E_x$.
\end{prop}

\begin{proof}
Let $x$ be from the set $\mc D(f)$.
Recall that then $E_x=P_x$ is a closed subgroup of $G$ which contains $H_x$.
According to Proposition \ref{p:generalised Atkinson} the $T_f$-orbit closure of $(x,H_x)$ is compact.
In particular $E_x/H_x$ is compact.
The projection of the $T_f$-orbit is also compact and $T$-invariant, thus it equals $X$.
As $H_x\subseteq P_x$ the same holds for the $T_f$-orbit closure $C$ of $(x,e)$.

Let $(y,g_0)$ and $(y,g_1)$ belong to $C$.
As in the proof of Proposition \ref{p:generalised Atkinson} we follow that
both
\[
H_y = g_i\cdot H_x\cdot g_i^{-1} \subseteq g_i\cdot E_x \cdot g_i^{-1} \subseteq E_y,
\]
for $i=0,1$.
On the other hand, by compactness of the orbit closure $C$ one can choose a sequence $\{n_k\}_{k\geq 1}$ and $g$ in $G$ such that $T^{n_k}y\rightarrow x$ and $f(n_k,y)\cdot H_y\rightarrow g\cdot H_y$.
Again, by the same reasoning as before (the $f(n_k,y)$ converge to $g$ modulo $g_i\cdot E_x \cdot g_i^{-1}$) both
\[
g \cdot g_i\cdot E_x \cdot g_i^{-1}\cdot g^{-1} \subseteq E_x,
\]
for $i=0,1$.
Since $E_x$ is a group $g_1^{-1}\cdot g_0$ belongs to the normaliser $N(E_x)$ of $E_x$. The only thing left to prove is that every slice $C_y$ consists of a single left coset of $E_x$, i.e. $g_1^{-1}\cdot g_0\in E_x$.
This is done by a simple `cohomology' argument.
Since $C_x=E_x$ both the sequences $f(n_k,y) \cdot g_i \cdot E_x$ from above converge to $E_x$.
Let us define a `boundary function' on our countable set $\{y\}\cup\{T^{n_k}y\}_{k}$ by setting $b_0= g_0$ and choosing
\[
b_k \in f(n_k,y)\cdot g_0\cdot N(E_x)
\]
such that $b_k\rightarrow e$.
Then
\[
c_k= b_k^{-1} \cdot f(n_k,y) \cdot b_0 \in N(E_x)
\]
and both the sequences $c_k \cdot (b_0^{-1} \cdot g_0) \cdot E_x$ and $c_k \cdot (g_0^{-1}\cdot g_1)\cdot E_x$ are contained in $N(E_x)$ and converge to $E_x$.
As $E_x$ is normal in $N(E_x)$ there exists a left-invariant metric for the topology in $N(E_x)/E_x$ and it follows that $(b_0^{-1}\cdot g_0)\cdot E_x= (b_0^{-1}\cdot g_1)\cdot E_x$.
Thus $g_0 \cdot E_x= g_1 \cdot E_x$.
\end{proof}

\begin{cor}\label{c:normal subgroup}
Suppose that $T$ be a minimal rotation on a locally connected compact group $X$, and $f$ is a recurrent cocycle with values in a l.c.s.c. group $G$.
If there exists a point $x_0\in X$ for which the identity component of $E_{x_0}$ is a normal subgroup of $G$, then $f$ is strongly regular and $E_x / E_x^0$ is compact.
\end{cor}
\begin{proof}
We apply Proposition \ref{p:regular orbit closure} to the consistent selection defined by setting $H_y = E_x^0$ for all $y$ in $X$.
\end{proof}

\begin{cor}
Under the same assumption as Corollary \ref{c:normal subgroup}, if there exists a point $x_0$ for which $E_{x_0}=\{e\}$ then $f$ is a coboundary.
\end{cor}
\begin{proof}
By the previous corollary, the $T_f$-orbit closure $C$ of any point $(x,e)$ with $x\in\mc D(f)$ is regular and compact. Let us set $H=E_x$.
By regularity every vertical section $C_y= g_y\cdot H$ for some $g_y$ in $G$, and moreover all essential ranges are conjugate to $H$ (see Section \ref{s:basic facts}).
Since $E_{x_0}$ is trivial so must be $H$, and therefore the set $C$ projects injectively onto $X$. This implies that $C$ is the graph of a continuous function $b:X\rightarrow G$ and $b(Ty)= f(y)\cdot b(y)$ for every $y$ in $X$.
Thus $f(y) = b(Ty)\cdot b(y)^{-1}$ is a coboundary.
\end{proof}

\section{Regularity in general Lie groups}\label{s:Lie groups}

Throughout this section we will assume that $G$ is a connected Lie group, and $\mf G$ is its Lie algebra.
As usual, the group $\Aut(G)$ of all bicontinuous automorphisms of $G$ is considered as a (closed) subgroup of $\GL(\mf G)$.
We denote by $\Ad(G)$ the image of $G$ under the adjoint representation. Since $G$ is connected, $\Ad(G)$ is contained in $\Aut(G)^0$ the identity component of the automorphism group, which is an
almost algebraic subgroup of $\GL(\mf G)$ (i.e. of finite index in some algebraic subgroup of $\GL(\mf G)$; this is a theorem of D. Wigner, cf. \cite{Dani01}).

For any cocycle $f$ with values in $G$ we define its \emph{adjoint cocycle} by setting
\begin{equation*}
\Ad(f) (n,x) = \Ad\big(f(n,x)\big),
\end{equation*}
which is a cocycle taking values in $\Ad(G)\subseteq \GL(\mf G)$.
It is clear that if $f$ is continuous and recurrent so is $\Ad(f)$.

The following proposition describes the behaviour of the identity component of an essential range under conjugation by the cocycle $f$.
Its proof essentially uses the locally closedness of the orbit of a connected subgroup $H$ under the action of an almost algebraic group of automorphisms.
From this point of view it does not contain much new compared to \cite{Dani}.

\begin{prop}\label{p:compact orbits mod stab}
Suppose $T$ be a minimal homeomorphism of a compact metric space $X$ and $f$ is a continuous cocycle with values in a connected Lie group $G$.
Choose any almost almost algebraic and closed subgroup $A$ in $\Aut(G)^0$ which contains $\Ad(G)$.
If $x$ is any point from $\mc D(f)$ and $I_A(H)=\{\alpha\in A: \alpha(H)=H\}$ is the stabiliser of the identity component $H=E_x^0$ in $A$, then
the orbit closure
\[
C^*=\overline{T_{Ad(f)}^\mathbb Z\big(x, I_A(H)\big)}
\]
taken in $X\times A/I_A(H)$ has the following properties:
\begin{enumerate}
\item  it is compact and the action of $T_{\Ad(f)}$ restricted to $C^*$ is minimal,
\item it projects onto $X$, and injectively onto the set $\mathcal D(f)$.
\end{enumerate}
In other words the system $(C^*,T_{Ad(f)})$ is an almost one-to-one
extension of $(X,T)$.
\end{prop}

\begin{proof}
Let $\mc H(\mf G)$ be the Grassmanian manifold of all subalgebras of our Lie algebra $\mf G$.
Let $x$ be as above, and $\mf H_x$ be the subalgebra that corresponds to the identity component $H_x= E_x^0$.
We choose open neighbourhoods $\mf U$ of $0$ in $\mf G$ and $U$ of $e$ in $G$ such that the exponential mapping is a diffeomorphism between $\mf U$ and $U$.
If $\{n_k\}_{k\geq 1}$ is any sequence of integers such that $T^{n_k}x\rightarrow y\in\mc D(f)$ then by compactness of $\mc H(\mf G)$ we have convergence (along some subsequence) of the conjugate subalgebras
\[
\mf H_k= \Ad\big(f(n_k,x)\big)\mf H_x \rightarrow \mf H',
\]
where $\mf H'$ is some subalgebra of the same dimension as $\mf H$. 
As $\exp(\mf H_k \cap \mf U) \subseteq E_{T^{n_k} x}^0 \cap U $ we conclude from semi-continuity of the essential ranges that $\exp(\mf H'\cap \mf U)\subseteq E_y\cap U$, and as $E_y$ is a closed group $E_y^0$ contains the closed subgroup generated by $\exp(\mf H')$.
Thus if we denote by $\mf H_y$ be the algebra corresponding to $H_y=E_y^0$, then $\mf H'\subseteq \mf H_y$.  
By the same reasoning, if $\{m_k\}_{k\geq 1}$ is such that $T^{m_k}y\rightarrow x$ we again may assume convergence (along some subsequence) of
\[
\Ad\big(f(m_k,y)\big)\mf H_y \rightarrow \mf H'',
\]
where $\mf H''$ is a subalgebra of the same dimension as $\mf H_y$, and that $E_x^0$ contains the closed subgroup generated by $\exp(\mf H'')$.
Therefore $\mf H''\subseteq \mf H_x$ and since $\mf H''$ has at least the dimension of $\mf H_x$ we conclude that $\mf H''=\mf H_x$ and also $\mf H' = \mf H_y$.
In other words, $\mf H_y$ is in the closure of the $A$-orbit of $\mf H_x$ and vice versa.
As $A$ is almost algebraic its orbits on $\mc H(\mf G)$ are locally closed \cite[Corollary 3.2.12]{Zimmer}, which is the same as saying that the factor map
\[
A/I_A(\mf H_x) \longrightarrow 
\mc H(\mf G),
\quad \alpha \cdot I_A(\mf H_x)\mapsto \alpha (\mf H_x),
\]
with $I_A(\mf H_x)= \big\{\alpha\in A: \alpha(\mf H_x)=\mf H_x \big\}$, is a homeomorphism between $A/I_A(\mf H_x)$ and the orbit $\mf H_x^A = \{\alpha(\mf H_x): \alpha \in A\}$, cf. \cite[Lemma 2.1.15]{Zimmer}.
We therefore follow that $\mf H_y$ must belong to $\mf H_x^A$, otherwise we contradict locally closedness of the $A$-orbits.
Hence $\mf H_y = \alpha_y (\mf H_x)$ for some $\alpha_y$ which is uniquely determined modulo $I_A(\mf H_x)$, and
\begin{gather*}
\Ad(f(n_k,x)) \cdot I_A(\mf H_x) \rightarrow \alpha_y \cdot I_A(\mf H_x)
\end{gather*}
along this subsequence of $\{n_k\}_{k\geq 1}$.
This means that $y$ is contained in the $\pi_X$-projection of the orbit closure
$C^*=\overline{T_{\Ad(f)}^\Z \big(x,I_A (\mf H_x)\big)}$.
Since $y$ in $\mc D(f)$ was chosen arbitrarily, the orbit closure $C^*$ projects onto $\mc D(f)$. 
By Lemma \ref{l:surjective} we can find a compact subset $K$ in $G$ such that 
\[
\big(X\times K \cdot I_A(\mf H_x)\big) \cap C^*
\]
projects onto the whole set $X$.
Since for every $y\in\mc D(f)$ the vertical section $C_y^*= \big\{ \alpha_y \cdot I_A(\mf H_x) \big\}$ is contained in the compact set $K\cdot I_A(\mf H_x)$ we conclude that the whole closure $C^*$ is contained in the compact set $X\times K\cdot I_A(\mf H_x)$.

Minimality of $C^*$ is clear since $T$ is minimal and the vertical section $C^*_x$ consists of a single point only.
\end{proof}

\begin{rem}\label{r:criterions}
It follows immediately from the above proof that on the comeager set $\mc D(f)$ all identity components $E_y^0$ are $A$-conjugate, i.e. for every $x,y\in\mc D(f)$, $E_y^0$ is the $A$-image of $E_x^0$.
\end{rem}

The connection of Proposition \ref{p:compact orbits mod stab} with a general regularity theorem as mentioned in the introduction is as follows:
If we could prove that the almost one-to-one extension $C^*$ in Proposition \ref{p:compact orbits mod stab} projects injectively onto the whole set $X$, then the mapping
\[
y\mapsto C^*_y = \alpha_y \cdot I_A(H)
\]
is continuous and therefore 
\[
H_y = \alpha_y (H)
\]
defines a consistent selection $\{H_y\}_{y\in X}$.
Thus if $T$ is a minimal rotation on a locally connected compact group $X$, we would be able to conclude with help of the generalised Atkinson's Proposition \ref{p:regular orbit closure} that every $f$ admits strongly regular orbit closures.
This makes the following open question so important for us:

\begin{open question}\label{no hairs}
Let $T_f$ be a continuous $G$-extension of a minimal group rotation $T$ (or more generally any minimal homeomorphism), and $H$ be a closed subgroup of $G$. 
Suppose $C\subseteq X\times G/H$ is a $T_f$-invariant compact set such that for every $x$ belonging to a dense $G_\delta$-set in $X$ the vertical section $C_x=\{gH\in G/H : (x,gH)\in
C\}$ consists of a single coset $g_x H$.
Is it true that then the same holds for every $x$ in $X$?
\end{open question}

This question can be answered positively for certain cases, as shown in \cite{nilpotent}.
For example, if for every $g\notin H$ we know that 
\[
e \notin \overline{HgH} = \overline{\{h_1\cdot g\cdot h_2 : h_1,h_2 \in H\}},
\]
which is always fulfilled in any nilpotent (or virtually nilpotent) group $G$ \cite[Theorem 3.1]{nilpotent}, or if $H$ is a normal subgroup of a (not necessarily nilpotent) group $G$.
However, it not clear to us whether the above open question is true in such a general formulation.

Now, let us provide a special version of Proposition \ref{p:compact orbits mod stab}, in which we replace the almost algebraic group $A$ by $\Ad(G)$ itself.
This version parallels the result on the identity components of the stabilisers of general Borel actions \cite[Corollary 5.3]{Dani}.

\begin{prop}\label{p:compact orbits mod normaliser}
Suppose $T$ be a minimal homeomorphism of a compact metric space $X$ and $f$ is a continuous cocycle with values in a connected Lie group $G$.
Let $H$ be the identity component of $E_x$ at some point $x\in \mc D(f)$, and $N(H)$ its normaliser in $G$.
Further assume that one of the following properties from \cite[Theorem 3.2]{Dani} are fulfilled, i.e.
\begin{enumerate}
\item $\Ad(G)$ is almost algebraic;
\item for all $g$ from the radical of $G$, the eigenvalues of $\Ad(g)$ are real;
\item $H$ is compact.
\end{enumerate}
Then the $T_f$-orbit closure $C^* = \overline{T_f^\Z \left(x,N(H)\right)}$ in $X\times G/N(H)$ is minimal and compact  and projects almost one-to-one onto $X$.
Besides for all $y$ in $\mc D(f)$, the identity component $E_y^0$ is conjugate to $H$.
\end{prop}

\begin{rem}
There are several criterions given in \cite[Proposition 3.4]{Dani} which guarantee that the group $\Ad(G)$ itself is almost algebraic, for example if $G$ is an almost algebraic subgroup of $\GL(n,\R)$ for $n\geq 2$ or $G$ is semi-simple, among others.
\end{rem}

\begin{proof}[Proof of Proposition \ref{p:compact orbits mod normaliser}]
By \cite[Theorem 3.2]{Dani}, if one of the three conditions is satisfied the conjugacy class $H^G=\{g\cdot H\cdot g^{-1}: g\in G\}$ is locally closed in the space $\mc C(G)$ of closed subgroups and therefore the map 
\[
G/N(H)\longrightarrow H^G,\quad g\cdot N(H)\mapsto g\cdot H \cdot g^{-1} 
\]
is a homeomorphism.
Using this fact - considering the adjoint action of $f$ on $\mc C(G)$ rather than on the Grassmanian $\mc H(G)$ - we conclude in the same manner\footnote{the only detail which has to be considered additionally is the semi-continuity of  dimension: if $g_k \cdot H \cdot g_k^{-1}\rightarrow H'$ with respect to the Fell topology, then $\dim H'\geq \dim H$ .
} as in the proof of Proposition \ref{p:compact orbits mod stab} that the orbit closure of $C^* = \overline{T_f^\Z \left(x,N(H)\right)}$ in $X\times G/N(H)$ is minimal, compact and projects injectively onto the comeager set $\mc D(f)$.
Furthermore, these properties of $C^*$ immediately imply the assertion on the identity components $E_y^0$, cf. also Remark \ref{r:criterions}.
\end{proof}

Proposition \ref{p:compact orbits mod normaliser} together with Proposition \ref{p:regular orbit closure} yields also an alternative proof of the regularity result \cite[Theorem 4.9]{nilpotent}.

\begin{cor}[\cite{nilpotent}, Theorem 4.9]
Let $T$ be a minimal rotation on a locally connected compact group $X$, and $G$ a connected nilpotent Lie group.
If $f$ is a continuous and recurrent cocycle with values in $G$ then $f$ is strongly regular.
Furthermore, all $E_x$ are conjugate and $E_x/E_x^0$ is compact.
\end{cor}
\begin{proof}
Every $\Ad(g)$ has real eigenvalues only (actually, all eigenvalues are equal to one) and satisfies condition (ii) from \cite[Theorem 3.2]{Dani} listed in \ref{p:compact orbits mod normaliser}. Thus the orbit closure $C^*$ of $\big(x,N(H)\big)$ is compact and projects onto $X$, whereas it projects injectively onto the set $\mc D(f)$.
As mentioned above $C^*$ must be a one-to-one extension of $X$ and so $H_{T^n x}= f(n,x)\cdot H \cdot f(n,x)^{-1}$ extends to a consistent selection of conjugate subgroups.
Now Proposition \ref{p:regular orbit closure} yields the assertion of the corollary.
\end{proof}

Another consequence of Proposition \ref{p:compact orbits mod normaliser} is
the following partial result on regularity which holds even for an arbitrary minimal compact system $(X,T)$.
\begin{cor}
Suppose that $G$ is a connected Lie-group which fulfills one of the properties listed in Proposition \ref{p:compact orbits mod normaliser}.
If for some point $x\in \mc D(f)$, the identity component $H=E_x^0$ equals its own normaliser in $G$, then the $T_f$-orbit closure of $(x,e)$ is surjective and hence $f$ is regular.
\end{cor}

\begin{proof}
The assertion of the corollary is evident from Proposition \ref{p:compact orbits mod normaliser}, since for every $x\in\mc D(f)$ we have $N(H)=H\subseteq P_x$.
\end{proof}

\section{Regularity results for $\R^d\rtimes \R$}

Let $\R$ act continuously by linear automorphisms $A_u$ ($u\in \R$) on $\R^{d}$,
and $G$ be the semi-direct product $G= \R^d \rtimes \R$ defined by the group operation
\[
(v_1, u_1) \cdot (v_2, u_2) = \left(v_1 + A_{u_1}(v_2), u_1+u_2 \right).
\]
With this definition the sets
\[
U= \{e\}\times \R \quad\text{and}\quad N=\R^d\times \{e\}
\]
are subgroups of $G$, with $N$ normal in $G$, and conjugation by $u$ in $U$ equals the automorphism $A_u$ on $N$.
Let
\[
\pi: G=\R^d\rtimes \R \longrightarrow \R
\]
denote the projection of $G$ onto its second coordinate, and denote by
\[
\pi(f) (n,x) = \pi\big(f(n,x)\big)
\]
the factor cocycle with values in $\R$.

In spite of Question \ref{no hairs} remains open even for this special group, we are able to prove the existence of surjective orbit closures as the following theorem shows.
Its proof involves a direct computation of compactness of the cocycle modulo the normaliser $N(H)$ of
the identity component $H= E_x^0$, and uses the simple group structure to reduce to the case where $H$ equals its own stabiliser.

\begin{theo}\label{t:semidirect product}
Let $f$ be a continuous and recurrent cocycle over a minimal rotation on a locally connected compact group $X$ with values in the semi-direct product $G= \R^d \rtimes \R$.
If the action of $\R$ on $\R^d$ has no eigenvalue equal to one\footnote{by which we mean that for every $u\in\R$ the transformation $A_u$ has no eigenvalue equal to one}, then $f$ is regular.
\end{theo}

\begin{rem}\label{r:action}
Assuming the action of $\R$ has no eigenvalue equal to one implies (but is not equivalent to) the following local property: 
Let $\mf G$ the Lie algebra of $G$ and $\mf N$ be the subalgebra which corresponds to the normal abelian kernel $N$.
Then for every vector $\mf h$ which is not contained in the ideal $\mf N$,
\begin{equation*}\label{e:lie bracket}
\left[\mf h, \mf n\right] = \ad(\mf h,\mf n) \neq 0,
\end{equation*}
for all $\mf n$ in $\mf N$.
\end{rem}


\begin{proof}
\emph{Step 1.} Let $x$ be any point from our non-meager set
$\mc D= \mc D(f)\cap\mc D\big(\pi (f)\big)$,
and let $S$ be the essential range of the projected cocycle $\pi(f)$ at the
point $x$. Then the inverse image
\[
A=\pi^{-1}\big(\overline{T_{\pi(f)}^\Z (x,e)}\big)
\]
of the regular orbit closure of $(x,e)$ with respect of the projected cocycle is regular in the sense that every slice $A_y=\{g\in G: (x,g) \in A\}$ of $A$ consists of  a single coset $g_y\cdot \pi^{-1}(S)$, and further the map
\[
X\longrightarrow G/\pi^{-1}(S),  \quad y\mapsto A_y = g_y \cdot \pi^{-1}(S),
\]
is continuous.
For every $g$ in $\pi^{-1}(S)$ we can find a sequence $\{n_k\}_{k\geq 1}$ and $v_k\in N$ such that $T^{n_k}x \rightarrow x$ and $f(n_k,x)\cdot v_k \rightarrow g$.
Thus
\[
\begin{aligned}
E_{T^{n_k} x}\cap N &=  f(n_k,x) \cdot (E_x\cap N)\cdot f(n_k,x)^{-1} = \\ &=
 f(n_k,x)\cdot v_k \cdot (E_x\cap N)\cdot v_k^{-1} \cdot f(n_k,x)^{-1}
\end{aligned}
\]
since $N$ is abelian; by letting $k\rightarrow\infty$ it follows that
\[
E_x\cap N\supseteq g \cdot(E_x\cap N)\cdot g^{-1}.
\]
Thus $\pi^{-1}(S)$ is contained in the normaliser $N(E_x\cap N)$ and the map $y\mapsto g_y\cdot N(E_x\cap N)$ is continuous.
Use this map to define a consistent selection $\{N_y\}_{y\in X}$ of subgroups conjugate to $N_x=
E_x\cap N$ by setting
\[
N_y = g_y\cdot(E_x\cap N)\cdot g_y^{-1}.
\]
It is important to note that by symmetry\footnote{we could have started with any other $y$ in $\mc D$.} $N_y= E_y\cap N$ for all $y$ from our comeager set $\mc D$.

-------------------------------

\emph{Step 2.} We let $H=E_x^0$ be the identity component of $E_x$ and $\hat H=N(H)^0$ the identity component of the normaliser $N(H)$ and claim that
\[
C= \overline{T_f^\Z \big(x, \hat H\cap N\big)}
\]
projects onto the whole space $X$.
Let $y$ be any point in $\mc D$ and choose $T^{n_k}x$ converging to $y$ so that $f(n_k,x)= g_k\cdot v_k$, with $g_k\rightarrow g$ and $v_k$ in the kernel $N=\ker(\pi)$.
We denote by $\mf H$ and $\hat{\mf H}$ the subalgebras that correspond to $H$ and $\hat H$.
The conjugate subgroups
\[
E_{T^{n_k}x}= f(n_k,x)\cdot H \cdot f(n_k,x)^{-1} = g_k\cdot v_k  \cdot H \cdot v_k^{-1}\cdot g_k^{-1}
\]
correspond to the subalgebras
\[
\Ad\left(f(n_k,x)\right) \mf H = \Ad(g_k) \Ad(v_k) \mf H.
\]

For any $\mf v$ from the subalgebra $\mf N$ corresponding to $N$ we know
that $\ad(\mf v) (\;\cdot\;) \in \mf N$, since $N$ is normal. Besides
$\ad(\mf v) (\;\cdot\;)= 0$ on $\mf N$ as $N$ is abelian. Thus $\ad(\mf
v)^j = 0$ for all $j\geq 2$ and one can calculate
\[
\Ad(v) = \exp\left(\ad(\mf v)\right) (\mf h)= \sum_{k=0}^\infty \frac{1}{k!} \ad(\mf v)^k = 1 + \ad(\mf v),
\]
where $v= \exp(\mf v)$.
This implies that
\[
\Ad(f(n_k,x)) \mf H
= \Ad(g_k)( 1 + \underbrace{\left[\mf v_k, \;\cdot\; \right]}_{\in \mf N}) \mf H,
\]
with any choice of $\mf v_k\in\mf N$ such that $v_k=\exp(\mf v_k)$.
Note that since $\mf H \cap \mf N = \mf N_x$, where $\mf N_x$ is the subalgebra that corresponds to $N_x$, we have
\[
\hat{\mf H}\cap\mf N = \big\{ \mf v\in\mf N : \left[\mf v, \mf H\right] \subseteq \mf N_x \big\}.
\]
Assume for a moment that the $\mf v_k + (\hat{\mf H}\cap \mf N)$ are unbounded in $\mf N/(\hat{\mf H}\cap \mf N)$.
Then we can find\footnote{
choose any linear functional $\Lambda:\mf N \longrightarrow \R$ such that $\ker \Lambda =\mf N_x$. Then every $\mf v$ in $\mf N$ defines a linear functional $\Lambda_{\mf v}$ on $\mf H$ by putting $\Lambda_{\mf v}(\mf h)=\Lambda\big(\left[\mf v, \mf h\right]\big)$.
Then $\hat{\mf H}\cap\mf N$ is the kernel of the linear map $\mf v\mapsto\Lambda_{\mf v}$.
As $\mf H$ is finite-dimensional boundedness of the $\mf v_k$ modulo $\hat{\mf H}\cap\mf N$ is equivalent to
boundedness of  the $\Lambda_{\mf v_k}(\mf h)$ for every $\mf h$ in $\mf H$.
}
a vector $\mf h$ in $\mf H$ such that along some subsequence
\[
\left[ \mf v_k, \mf h \right] + \mf N_x \rightarrow\infty
\]
in the quotient space $\mf N/\mf N_x$.
This implies that the one-dimensional spaces
\[
\big(1 + \left[\mf v_k, \cdot\; \right]\big)\big(\langle\mf h\rangle\big)
\]
converge to some one-dimensional space $\langle \mf h'\rangle$ contained in $\mf N$ but not in $\mf N_x$.
As the $\{N_y\}_{y\in X}$ form a consistent selection,
\[
\Ad(g)(\mf N_x)  = \lim_k \Ad(g_k) (\mf N_x) = \lim _k \Ad\big(f(n,x)\big)(\mf N_x) = \mf N_y,
\]
and the subspaces
\[
\Ad(g_k) \big(1 + \left[\mf v_k, \cdot\; \right]\big)\big(\langle\mf h\rangle\big)
\]
converge to the one-dimensional subspace $\Ad(g)\big(\langle\mf h'\rangle\big)$ which is contained in $\mf N$ but not in $\mf N_y$.
By semi-continuity of the essential ranges the immersed subgroup corresponding to this one-dimensional subspace is contained in $E_y$ but not in $N_y$ (as in the proof of Proposition \ref{p:compact orbits mod stab}, this follows easily from the fact that the exponential mapping is a local diffeomorphism). 
We therefore contradict the fact that $N_y= E_y\cap N$.
Thus the $\mf v_k + (\hat{\mf H} \cap \mf N)$ stay in some compactum and the same is true for the
$v_k\cdot (\hat H\cap N)$.
This proves that the $T_f$-orbit closure modulo $\hat H\cap N$ projects onto $\mc D$ and Lemma \ref{l:surjective} shows that it projects onto the whole space $X$.

------------------------------

\emph{Step 3.}
Now, we distinguish two cases:
If $H$ is contained in the normal subgroup $N$, then $N_x = H = E_x^0$ and there exists a cutting neigbourhood for $N_x$ in $E_x$.
Applying Proposition \ref{p:generalised Atkinson} yields the existence of surjective closures.

If $H$ is not contained in $N$, then there exist a $\mf h\in\mf H$ outside $\mf N$.
By Remark \ref{r:action} the linear transformation $[\mf h ,\;\cdot \;]$ maps $\mf N$ bijectively onto itself; and the same is true for the invariant subspace $\mf N_x$.
Thus for any $\mf  v\in\mf N$ outside $\mf N_x$ we must have also $[\mf h,\mf v]\notin\mf N_x$ and so $\mf v\notin \hat{\mf H}$.
Therefore $\hat H\cap N = H\cap N$ and Step 2 together with the fact that $H\subseteq P_x$ yields that the $T_f$-orbit closure of $(x,e)$ is surjective.
\end{proof}

\end{document}